\documentclass[11pt, reqno, oneside, letterpaper]{article}

\usepackage[margin=1in]{geometry}
\usepackage{amsmath,amsthm,caption,graphicx,hyperref,makecell,mathtools,multirow,ragged2e,parskip,scalerel,tabularx,times,upgreek,amsfonts}
\usepackage{authblk}

\newcommand\wh[1]{\hstretch{2}{\hat{\hstretch{.5}{#1}}}}

\newtheorem{theorem}{Theorem}

\newtheorem{lemma}{Lemma}[section]
\newtheorem{proposition}[lemma]{Proposition}
\newtheorem{corollary}[lemma]{Corollary}

\renewcommand{\pi}{\uppi}
\renewcommand{\rho}{\uprho}
\renewcommand{\gamma}{\upgamma}
\renewcommand{\nu}{\upnu}

\newcommand{\R}{\mathbb{R}}

\renewcommand{\P}{\mathbf{P}}
\newcommand{\E}{\mathbf{E}}
\DeclareMathOperator{\1}{\mathbf{1}}
\newcommand{\eps}{\epsilon}

\DeclareMathOperator{\Var}{\mathbf{Var}}
\DeclareMathOperator{\Cov}{\mathbf{Cov}}
\DeclareMathOperator{\Corr}{\mathbf{Corr}}
\DeclareMathOperator{\Err}{\mathbf{L}}
\DeclareMathOperator{\MSE}{\mathbf{R}}
\DeclareMathOperator{\Prob}{\mathrm{Prob}}
\DeclareMathSymbol{\shortminus}{\mathbin}{AMSa}{"39}

\title{A local-global principle for nonequilibrium steady states}

\author{Jacob Calvert\thanks{calvert@gatech.edu} }
\author{Dana Randall\thanks{randall@cc.gatech.edu}}
\affil{School of Computer Science, Georgia Institute of Technology}

\date{}

\begin{document}







\maketitle

\thispagestyle{empty}
\begin{abstract}
\noindent 
The global steady state of a system in thermal equilibrium exponentially favors configurations with lesser energy. This principle is a powerful explanation of self-organization because energy is a local property of a configuration. For nonequilibrium systems, there is no such property for which an analogous principle holds, hence no common explanation of the diverse forms of self-organization they exhibit. However, a flurry of recent work demonstrates that a local property of configurations called ``rattling'' predicts the global steady states of a broad class of nonequilibrium systems. We interpret this emerging physical theory in terms of Markov processes to generalize and prove its main claims. Surprisingly, we find that the idea at the core of rattling theory is so general as to apply to equilibrium and nonequilibrium systems alike. Its predictions hold to an extent determined by the relative variance of, and correlation between, the local and global ``parts'' of an arbitrary steady state. We show how these key quantities characterize the local-global relationships of random walks on random graphs, various spin-glass dynamics, and models of animal collective behavior.
\end{abstract}


\newpage
\setcounter{page}{1}
\section{Introduction}

Self-organization abounds in nature, from the spontaneous assembly of fractal protein complexes \cite{sendker_emergence_2024}, to nests of army ants, or bivouacs, formed by entangling potentially millions of their own, living bodies \cite{daniel_j_c_kronauer_army_2020}, and ``marching bands'' of desert locusts that can span hundreds of square kilometers \cite{buhl_disorder_2006}. These varied phenomena warrant a common explanation because each entails a system that occupies relatively few configurations\footnote{We use ``configuration'' to refer to what is typically called a ``microstate,'' to avoid confusion with the states of a Markov chain.} among an overwhelming number of alternatives. For physical systems in thermal equilibrium, the Boltzmann distribution provides such an explanation: systems preferentially occupy relatively few configurations with lower energy. Specifically, the probability of a configuration $x$ of a physical system in thermal equilibrium satisfies
\begin{equation}\label{bg}
\Prob(x) \propto e^{-\beta E(x)}
\end{equation}
in terms of a constant $\beta$ and the configuration's energy $E(x)$ \cite{sethna_statistical_2021}.

The Boltzmann distribution is remarkable because the energy of a configuration is ``local'' in the sense that it does not depend on the dynamics that connect other configurations. This fact makes it possible to understand the equilibrium configurations of many-body systems, like proteins, far more efficiently than is possible with methods that apply to nonequilibrium systems \cite{noe_boltzmann_2019}. In contrast, there can be no local property of configurations that generally determines their weight in the global, steady-state distribution of a nonequilibrium system \cite{landauer_inadequacy_1975,andresen_objections_1984,landauer_motion_1988}. Progress therefore requires either characterizing configuration weights in nonlocal terms or identifying special classes of nonequilibrium systems for which a local characterization does apply. The emerging theory of ``rattling,'' a far-reaching principle of nonequilibrium self-organization that was introduced by Chvykov and England \cite{chvykov_least-rattling_2018}, takes the latter approach and has found many applications \cite{chvykov_low_2021,gold_self-organized_2021,jackson_emergent_2021,yang_emergent_2022,england_self-organized_2022,kedia_drive-specific_2023}.

\subsection{Rattling theory} According to Chvykov et al. \cite{chvykov_low_2021}, if the motion of a system is ``so complex, nonlinear, and high-dimensional that no global symmetry or constraint can be found for its simplification,'' then it amounts to diffusion in an abstract state space. Their heuristic then predicts that a particular function of the motion's effective diffusivity in the vicinity of a configuration $x$, called the rattling, determines the weight of~$x$ in the steady-state distribution. As a simple example, consider a one-dimensional diffusion with variable diffusivity $D(x)$, which has a steady-state distribution $\Prob(x)$ that is proportional to $1/D(x)$. In this case, Chvykov et al.\ define the rattling $R(x)$ to be $\log D(x)$ and their prediction is literally true: the steady-state distribution has the form of the Boltzmann distribution with $\beta = 1$, written in terms of $R(x)$ as 
\begin{equation}\label{rattling boltzmann}
    \Prob (x) \propto e^{-\beta R (x)}.
\end{equation}

To apply rattling to more general dynamics, Chvykov et al.\ replace $D(x)$ with an estimate of the rate at which the mean squared displacement from $x$ grows in each dimension and the rattling $R(x)$ with~$(\log \det D(x))/2$, which is a higher-dimensional analogue of $\log D(x)$ \cite{chvykov_low_2021}. Moreover, instead of the strict proportionality in Eq.\ \ref{rattling boltzmann}, which implies that the logarithms of $\Prob(x)$ and $1/R(x)$ are perfectly linearly correlated with a slope of $\beta = 1$, their theory predicts only that the correlation $\rho$ is high and that the slope $\beta$, which may vary by system, is ``of order $1$'' \cite{chvykov_low_2021}. Qualitatively, they predict that {\em systems spend more time in configurations that they exit more slowly}. This claim may seem trivial at first, but it is not---the steady-state distribution is global, while the rattling is local---and it cannot be true in general.

Although the initial body of work on rattling does not identify the precise assumptions required for this prediction to hold, recent work demonstrates the breadth of its applicability, including experiments with swarms of robots \cite{chvykov_low_2021} and collectives of active microparticles \cite{yang_emergent_2022}; simulations of spin glasses \cite{gold_self-organized_2021} and stochastic dynamics with strong timescale separation \cite{chvykov_least-rattling_2018}; and numerical studies of the Lorenz equations \cite{jackson_emergent_2021} and mechanical networks \cite{kedia_drive-specific_2023}. Strikingly, in many cases, the slope $\beta$ in Eq.\ \ref{rattling boltzmann} is nearly equal to $1$ \cite{chvykov_low_2021}.

\subsection{Summary of our results} The remarkable accuracy and scope of rattling theory's predictions suggest that it has a simple, underlying mathematical explanation which, until now, has been missing. Our main results provide such an explanation in terms of Markov chains, which serve as a general model of nonequilibrium steady states \cite{schnakenberg_network_1976,crooks_path-ensemble_2000,jiang_mathematical_2004,zia_probability_2007,esposito_three_2010,seifert_stochastic_2012}. Specifically, we derive formulas for Markov chain analogues of the correlation $\rho$ and slope $\beta$ in terms of two further quantities, denoted $\wh \rho$ and $r$, that respectively characterize the correlation between the stationary distribution's local and global ``parts'' and the relative number of exponential scales over which these parts vary. The quantities $\wh\rho$ and $r$ together determine the accuracy of rattling theory's predictions for any Markov chain.\footnote{Here and throughout, by ``Markov chain,'' we mean an irreducible, continuous-time Markov chain on a finite state space. Every such Markov chain has a unique stationary distribution.}

Even when the rattling heuristic holds, in the sense that $\rho$ and $\beta$ are close to $1$, the steady-state probabilities of some states can vastly differ from those predicted by Eq.\ \ref{rattling boltzmann}. This discrepancy may limit the value of the rattling heuristic to some applications, like the estimation of functionals of the steady-state distribution. With this in mind, we also consider a stronger claim than the rattling heuristic makes: that the ratios of the steady-state probabilities and those in Eq.\ \ref{rattling boltzmann} are uniformly close to $1$. Our results identify conditions on the rates $Q(x,y)$ of a Markov chain that guarantee that its stationary probabilities $\pi(x)$ are within a small constant factor of those specified by an analogue of Eq.\ \ref{rattling boltzmann}: 
\begin{equation}\label{mc rattling}
\pi(x) \propto q(x)^{-\beta}.
\end{equation}
Here, $q (x) = \sum_{y \neq x} Q(x,y)$ denotes the exit rate of state~$x$, the logarithm of which is the Markov chain analogue of rattling (see Appendix~\ref{appendix a} for details). When Eq.\ \ref{mc rattling} holds, $\pi$ is akin to a Boltzmann distribution because, like the energy of a configuration, the exit rate of a state is ``local'' in the right sense (Fig.~\ref{experiment figure}). We show that for Eq.\ \ref{mc rattling} to hold up to a small constant factor, it suffices for the global part of the stationary distribution to vary little.

\begin{figure}
\centering
\includegraphics[width=11.4cm]{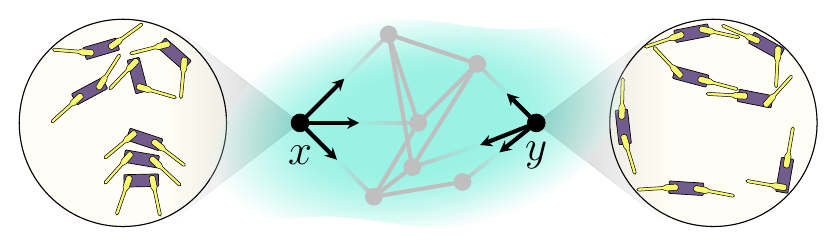}
\caption{Exit rates are local. A state $x$ of a Markov chain abstracts a configuration of a physical system, like the robot swarm of \cite{chvykov_low_2021}. The relative probability of $x$ in the stationary distribution is purely a function of the rates of the chain. The sum of the rates leaving $x$ is experimentally accessible, because this sum is the reciprocal of the average amount of time that it takes the system to leave $x$. This quantity is ``local'' to $x$ because it does not depend on the rates leaving any $y \neq x$. 
}
\label{experiment figure}
\end{figure}

\subsection{Applications} 

We demonstrate many potential applications across disciplines where the weights~$q(x)^{-\beta}$ can be used in place of the Boltzmann weights~$e^{-\beta E(x)}$, which only make sense in equilibrium settings. For example, Boltzmann generators are deep neural networks that are trained to generate unbiased one-shot samples from the equilibrium distribution of many-body systems, like proteins \cite{noe_boltzmann_2019}. The training data consist of pairs of configurations and their Boltzmann weights. Critically, these data can be efficiently generated because the weights are local functions of the configurations. By substituting these weights with Boltzmann-like weights, it may be possible to sample classes of nonequilibrium steady states far more efficiently than is possible with existing methods, like Markov chain Monte Carlo or molecular dynamics.

A second category of applications concerns the use of the quantities $r$ and $\wh\rho$ to characterize the way that driving affects the relationship between the local and global parts of the resulting nonequilibrium steady state. By ``driving,'' we mean the work done on a system by forces that are patterned in space and time. In stochastic thermodynamics, driven systems are often modeled by Markov chains with rates that are functions of a control parameter. For example, this is the context of the Jarzynski equality \cite{jarzynski_equilibrium_1997} and the Crooks fluctuation theorem \cite{crooks_nonequilibrium_1998,crooks_entropy_1999}. Because we have not precisely defined $r$ and $\wh\rho$ yet, we only explain the high-level idea. Briefly, every Markov chain defines a pair $(r,\wh\rho)$ in the strip~$[0, \infty)~\times~[-1,1]$, and driving a system amounts to placing a dynamics on these pairs. Our main results allow us to interpret certain regions of the strip as places where nonequilibrium steady states can be expressed with local weights. Analyzing driving in this way may improve our understanding of dissipative self-organization \cite{england_statistical_2013,england_dissipative_2015}.

\subsection{Broader perspective and related work} Markov chains that satisfy detailed balance are analogous to equilibrium systems because they are statistically time-reversible and have stationary distributions that can be expressed in the Boltzmann form \cite{zia_probability_2007}.\footnote{A Markov chain satisfies detailed balance if there is no net probability flux between any two states, i.e., if $\pi(x) Q(x,y) = \pi(y) Q(y,x)$ for all $x$ and $y$.} Accordingly, studies that model nonequilibrium steady states as Markov chains emphasize how and to what extent these chains violate detailed balance. For example, the conservation of probability requires that violations of detailed balance arise from net fluxes of probability around cycles of states. These circular fluxes form the basis of Hill's thermodynamic formalism \cite{hill_studies_1966,hill_free_1977}, Schnakenberg's network theory \cite{schnakenberg_network_1976,schnakenberg_thermodynamic_1981}, and cycle representations of Markov chains \cite{macqueen_circuit_1981,minping_circulation_1982,kalpazidou_cycle_2006,altaner_network_2012}, and they appear in the Helmholtz--Hodge decomposition of Markov chains \cite{strang_applications_2020}. Moreover, probability fluxes---and functions of them, like thermodynamic forces and entropy production---are the subject of nonequilibrium fluctuation theorems \cite{lebowitz_gallavotticohen-type_1999,qian_nonequilibrium_2001,andrieux_fluctuation_2007,bertini_flows_2015}, thermodynamic uncertainty relations \cite{barato_thermodynamic_2015,gingrich_dissipation_2016,horowitz_thermodynamic_2020}, and results concerning the response of nonequilibrium steady states to rate perturbations \cite{owen_universal_2020,owen_size_2023,fernandes_martins_topologically_2023,aslyamov_nonequilibrium_2024}.

It may come as a surprise, then, that our formulation of rattling theory in terms of Markov chains makes no reference to detailed balance or the nature of its possible failure. Instead, we view the search for nonequilibrium analogues of the Boltzmann distribution as part of a broader effort to understand when local information about the rates of a Markov chain suffices to determine or predict its stationary distribution (Fig.~\ref{experiment figure}). In this view, chains that satisfy detailed balance are not an ideal from which other chains depart; while their stationary distributions can be expressed in the Boltzmann form, the associated weights are nonlocal functions of the rates \cite{zia_probability_2007,altaner_network_2012} (Fig.~\ref{mctt figure}).

In general, if $Q$ is an irreducible rate matrix, then its stationary distribution $\pi$ is the unique solution to~$\pi Q = 0$, which implicates all of the rates $Q(u,v)$. In fact, according to the Markov chain tree theorem \cite{leighton_estimating_1986}, $\pi(x)$ is proportional to a sum over spanning trees of products of rates:
\begin{equation}\label{mctt}
    \pi(x) \propto \sum_{\mathcal{T}} \prod_{(u,v) \in \mathcal{T}_x} Q(u,v). 
\end{equation}
Here, the sum ranges over spanning trees $\mathcal{T}$ of the chain's transition graph and $\mathcal{T}_x$ is the set of directed edges that results from ``pointing'' all of the edges of $\mathcal{T}$ toward $x$ (Fig.~\ref{mctt figure}). Our results identify special classes of chains for which relatively little and, importantly, exclusively local information about their rates suffices to determine their stationary distributions. For example, for chains that satisfy Eq.\ \ref{mc rattling}, it suffices to know the exit rates $q(x)$ to determine $\pi$.

\begin{figure}
\centering
\includegraphics[width=11.4cm]{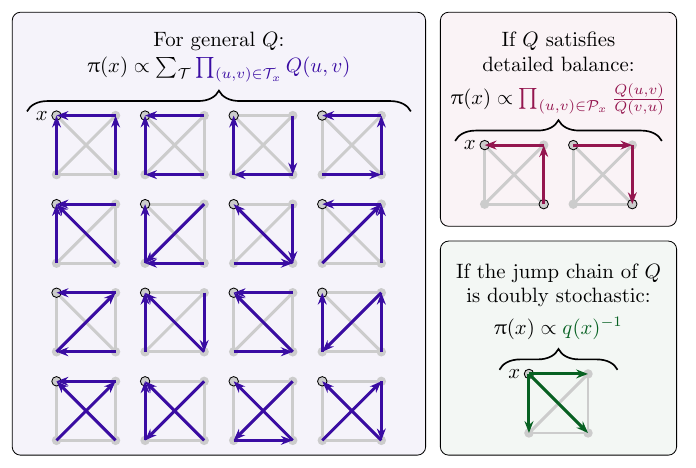}
\caption{The relationship between the rates of an irreducible Markov chain $Q$ and its stationary distribution $\pi$ specializes for certain classes of rate matrices. In general, all of the rates are necessary to determine $\pi (x)$; the Markov chain tree theorem specifies how to do so, by summing products of the rates along the spanning trees of the underlying adjacency graph (Eq.\ \ref{mctt}). However, if the chain satisfies detailed balance, then it suffices to know the ratio of the product of forward rates to the product of reverse rates along a path $\mathcal{P}_x$ to $x$ from a fixed, reference state. If the chain has a doubly-stochastic jump chain, then it suffices to know the exit rate $q(x)$ of each state $x$.}
\label{mctt figure}
\end{figure}

\section{Results}\label{sec:results}

We model the steady state of a physical system as a stationary, continuous-time Markov chain \cite{schnakenberg_network_1976,jiang_mathematical_2004,zia_probability_2007}. For simplicity, we assume that the state space has a finite number of states~$N$, and we refer to the chain by the matrix $Q$ of its transition rates. We further assume that the chain is irreducible, meaning that it can eventually reach any state from any other state; this guarantees that $Q$ has a unique stationary distribution $\pi$.

Our results emphasize the relationship between $\pi$, the exit rates
\[
q(x) := \sum_{y \neq x} Q(x,y),
\] 
and the stationary distribution of a chain that is closely related to $Q$, called the jump chain of $Q$. The jump chain is part of a standard construction of continuous-time Markov chains \cite[Section 2.6]{norris_markov_1997}. We define it to be the discrete-time Markov chain with transition probabilities 
\begin{equation*}\label{eq:jump chain tps}
\wh Q(x,y) := 
\begin{cases}
    Q(x,y)/q(x) & x \neq y,\\ 
    0 & x = y.
\end{cases}
\end{equation*}
The jump chain $\wh{Q}$ is irreducible because $Q$ is, hence it has a unique stationary distribution $\wh \pi$.

We view the logarithms of $1/q$ and $\wh \pi$ as the local and global ``parts'' of $\pi$. In this view, Markov chain analogues of the Boltzmann distribution (Eq.\ \ref{bg}) and the prediction of rattling theory (Eq.\ \ref{rattling boltzmann}) would approximate $\pi$ by a distribution with weights that are a function of $q$, with an error that depends on the relationship between $q$ and $\wh \pi$. Indeed, a simple heuristic suggests that the analogue of rattling $R$ is $\log q$ (see Appendix~\ref{appendix a}). Eq.\ \ref{rattling boltzmann} therefore suggests that we approximate $\pi(x)$ by 
\[
    \nu_\beta (x) := \frac{q(x)^{-\beta}}{\sum_y q(y)^{-\beta}}
\] 
for a real number $\beta$. Informally, our results identify two distinct ways for $\nu_\beta$ to approximate $\pi$ for some $\beta$, i.e., for $\pi$ to have local weights:
\begin{enumerate}
    \item The global part of $\pi$ is approximately uniform.
    \item The local and global parts of $\pi$ are approximately collinear.
\end{enumerate}

The first way for $\pi$ to have local weights is for $\wh \pi$ to be close to uniform. To state this more precisely, we will say that two numbers $a$ and $b$ are {\em within a factor of $k \geq 1$} if
\[
    k^{-1} \leq \frac{a}{b} \leq k.
\]
Additionally, we will say that two vectors or matrices $A$ and~$B$ of the same size are within a factor of $k \geq 1$ if all of their corresponding entries $A(x,y)$ and $B(x,y)$ are within a factor of $k$.

\begin{theorem}\label{thm a}
    If $\wh\pi$ is within a factor of $k \geq 1$ of the uniform distribution, then $\pi$ and $\nu_1$ are within a factor of~$k^2$. In particular, if $\wh\pi$ is uniform, then $\pi = \nu_1$.
\end{theorem}

In fact, $\wh\pi$ is uniform if and only if the jump chain $\wh Q$ is doubly stochastic, i.e., each of its rows and columns sums to~$1$. Since small perturbations of $\wh Q$ produce correspondingly small perturbations of $\wh\pi$ \cite{ocinneide_entrywise_1993,thiede_sharp_2015}, the conclusion of Theorem~\ref{thm a} applies whenever the original chain $Q$ has a jump chain that is nearly doubly stochastic. 

\begin{theorem}\label{thm b}
    If there is a doubly-stochastic matrix $D$ that is within a factor of $k \geq 1$ of the jump chain's transition probability matrix $\wh Q$, then $\pi$ and $\nu_1$ are within a factor of~$k^{2N}$. In particular, if $\wh Q$ is doubly stochastic, then $\pi = \nu_1$.
\end{theorem}

Theorem~\ref{thm b} makes it easy to generate examples of nonequilibrium steady states with local weights (Fig.~\ref{ds example figure}). Simply multiply the rows of a doubly-stochastic matrix by any exit rates $q(x) > 0$, and replace the diagonal entries with $-q(x)$. This yields a Markov chain $Q$ with local weights $-\!\log q(x)$ and the stationary distribution $\nu_1$.

\begin{figure}[t]
\centering
\includegraphics[width=\textwidth]{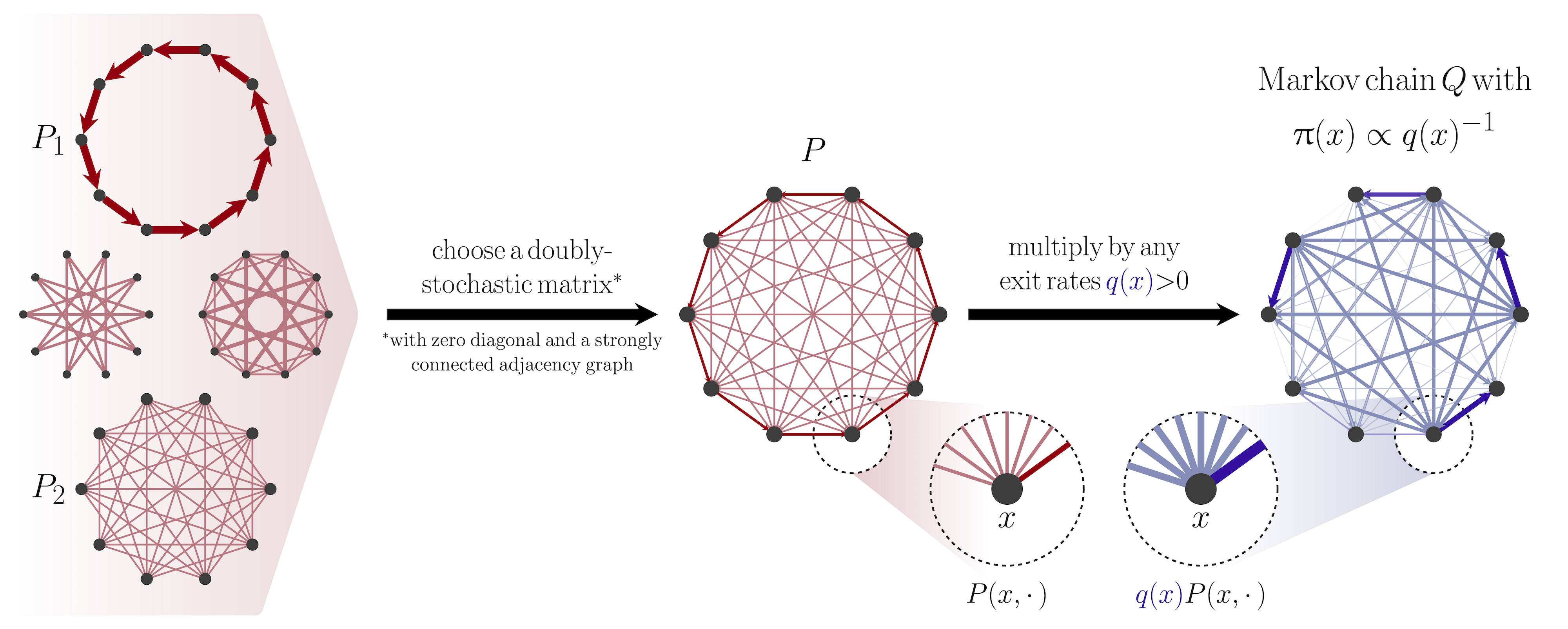}
\caption{A variety of Markov chains that satisfy $\pi (x) \propto 1/q(x)$. (Transition graphs with red edges depict doubly-stochastic matrices; edge widths are proportional to transition probabilities.) Doubly-stochastic matrices are varied, ranging from sparsely connected, directed cycles ($P_1$) to densely connected, undirected graphs ($P_2$). Moreover, they include arbitrary convex combinations, like $P$, which is a convex combination of $P_1$ and $P_2$. According to Theorems~\ref{thm a}~and~\ref{thm b}, scaling the rows $P(x,\cdot\,)$ by any positive numbers $q(x)$, and subsequently choosing the diagonal entries to make the row sums equal zero, produces the transition rate matrix of a continuous-time Markov chain~$Q(x,y) = q(x) P(x,y)$ that has a stationary distribution $\pi (x) \propto 1/q(x)$. (The blue edges have widths proportional to the corresponding rates of $Q$.)  This requires $P$ to have entries of zero on its diagonal (no ``self-loops'') and an adjacency graph that is strongly connected. The variety of doubly-stochastic matrices, and the freedom to scale their rows by any exit rates, means that an even greater variety of chains satisfy $\pi (x) \propto 1/q(x)$.}
\label{ds example figure}
\end{figure}

Even when $\wh\pi$ is far from uniform, the exit rates can ``predict'' the stationary probabilities, in the sense that the logarithms of $1/q(X)$ and $\pi (X)$ are highly collinear, for a random state $X$. This is true, for example, of many Markov chains with random rates, for which the slope of the linear relationship is nearly equal to $1$ (Fig.~\ref{correlation example figure}). Our next results explain why the correlations are high and the slopes are close to~$1$.

We measure the strength of collinearity using the linear correlation coefficient, defined for random variables~$U$ and $V$ that have finite, positive variances as
\[
    \Corr (U,V) := \frac{\E (UV) - \E(U) \E(V)}{\sqrt{\Var(U)\Var(V)}}.
\]
Specifically, we consider the correlation
\begin{equation}\label{corr def}
    \rho := \Corr \left( -\!\log q(X) , \log \pi (X)\right),
\end{equation}
for a random state $X \sim \P$. For example, $\P$ could be the uniform distribution over states or the stationary distribution~$\pi$. While the value of $\rho$ depends on the distribution of $X$, our results apply to any distribution. We assume only that $\pi (X)$ and $q(X)$ are non-constant, so that $\rho$ is defined.

Our next result is a formula for $\rho$ in terms the correlation between the local and global parts of $\pi$
\[
    \wh\rho := \Corr \left(-\!\log q(X), \log \wh\pi(X)\right),
\]
and the relative number of exponential scales over which they vary
\[
    r := \sqrt{\frac{\Var (\log \wh\pi(X))}{\Var (\log q(X))}}.
\]
Technically, $\wh\rho$ is undefined when $\wh\pi (X)$ is constant. However, in this case, we will adopt the useful convention that $\wh\rho = 0$.

\begin{theorem}\label{corr}
The correlation coefficient $\rho$ satisfies
\begin{equation}\label{main est}
\rho = \frac{1+\wh \rho r}{\sqrt{1 + 2 \wh \rho r + r^2}}.
\end{equation}
\end{theorem}

According to Eq.\ \ref{main est}, the correlation $\rho$ is strictly positive unless $\wh\rho$ is at most $-1/r$, which is impossible if $r$ is less than~$1$ because $\wh\rho \in [-1,1]$. Plots of the contours of $\rho$ in the $(r,\wh\rho)$ strip (Fig.~\ref{fig:b1}) further show that $\rho$ is close to $1$ to the extent that $\wh\rho$ is close to $1$ or $r$ is close to $0$. A simple lower bound of $\rho$, which makes no reference to $\wh\rho$, reflects the latter fact (Corollary~\ref{cor:a1}).

\begin{figure}[t]
\centering
\includegraphics[width=\textwidth]{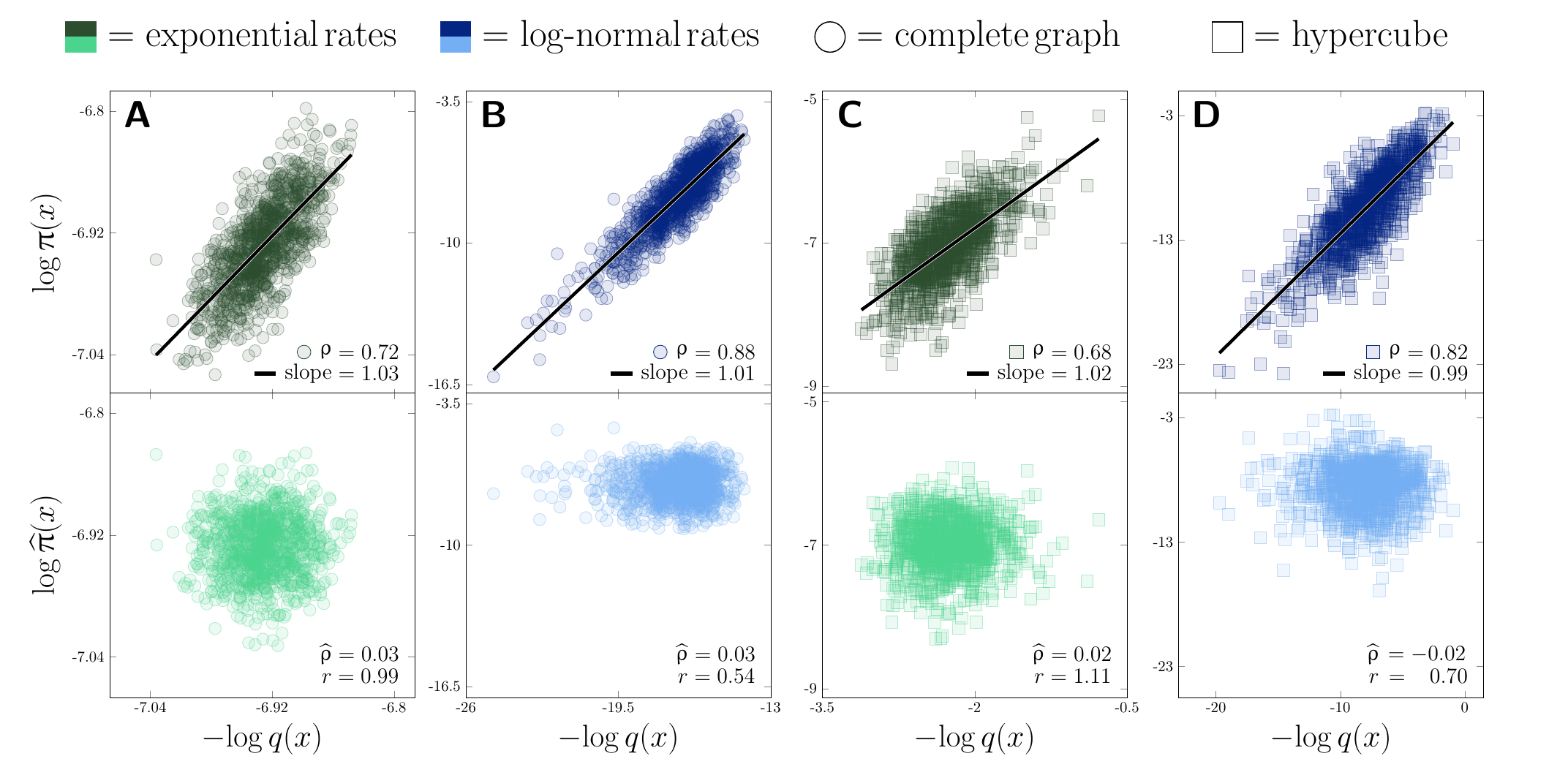}
\caption{Extent to which exit rates predict the stationary distributions of Markov chains with random transition rates. Each panel includes a scatter plot of $\log \pi(x)$ versus~$\shortminus\!\log q(x)$, paired with a corresponding plot of $\log \protect\wh{\pi} (x)$ versus~$\shortminus\!\log q(x)$, for one of four Markov chains $Q$. Each Markov chain has $2^{10}$ states that correspond to the vertices of either the complete graph (circular marks in A, B) or the $10$-dimensional hypercube (square marks in C, D). In (A, C), independent and identically distributed (i.i.d.) exponential rates $Q(x,y)$ and $Q(y,x)$ with mean $1$ connect every pair of states~$x,y$ that are adjacent in the underlying graph (green). In (B, D), the rates instead have i.i.d.\ log-normal distributions with parameters $\mu=0$ and $\sigma=5$ (blue). (A) Observe that, for exponential rates on the complete graph, the logarithms of $\pi (x)$ and $1/q(x)$ are highly collinear, while those of $\protect\wh\pi (x)$ and $1/q(x)$ are not. (B) The same is true of log-normal rates on the complete graph, but the collinearity of $\log \pi(x)$ and~$\shortminus\!\log q(x)$ is greater than in (A) because $\protect\wh\pi(x)$ varies over relatively fewer exponential scales than $q(x)$ in the case of log-normal rates. (In the language of Theorem~\ref{corr}, $r$ is smaller in B than in A.) (C, D) On the hypercube,~$\log \pi(x)$ and~$\shortminus\!\log q(x)$ remain highly collinear, but to a lesser extent than in (A, B), because the variance of $\protect\wh\pi(x)$ is greater relatively in (C, D). The slopes of the lines of best fit are nearly equal to $1$ in all cases, because the correlation $\protect\wh\rho$ between the local and global ``parts'' of $\pi$ is close to $0$ (see Theorem~\ref{mse}).}
\label{correlation example figure}
\end{figure}

Theorem~\ref{corr} explains the relatively high correlations exhibited by Markov chains with random rates (Fig.~\ref{correlation example figure}). Specifically, when the random rates are exponentially distributed, the local and global parts of $\pi$ are approximately uncorrelated and have nearly equal variances (Fig.~\ref{correlation example figure}A and \ref{correlation example figure}C), meaning that $\wh\rho \approx 0$ and $r \approx 1$, hence~$\rho \approx 1/\sqrt{2} \approx 0.71$ by Eq.\ \ref{main est}. When the rates instead have log-normal distributions, the local and global parts remain approximately uncorrelated, but the local part varies relatively more, leading to a lesser~$r$ and a correspondingly greater $\rho$  (Fig.~\ref{correlation example figure}B and \ref{correlation example figure}D). The difference between the cases of exponential and log-normal rates is not due to the exponential rates' lesser variance; the key quantities in Fig.~\ref{correlation example figure}A remain constant as the variance of the exponential rates increases over hundreds of scales (Fig.~\ref{fig:b2}). Note that the correlations in Fig.~\ref{correlation example figure} concern a uniformly random state $X$. If $X \sim \pi$ instead, then the experiments in Fig.~\ref{correlation example figure}A--C have qualitatively similar results, but the case of log-normal rates on the hypercube (Fig.~\ref{correlation example figure}D) is markedly different: the correlation $\wh\rho$ is $-0.40$, and the linear relationship between the logarithms of $1/q(X)$ and $\pi (X)$ has a lesser slope of $0.72$ (Fig.~\ref{fig:b3}).

In contrast to the conclusions of Theorems~\ref{thm a}~and~\ref{thm b}, there may be no $\beta$ for which $\pi$ and $\nu_\beta$ are within a small factor, even when $\rho$ is close to $\pm 1$. (See Proposition~\ref{prop:l1 l2 bd} for a bound on this factor.) However, in this case, there is a $\beta$ for which~$\nu_\beta$ accurately predicts the relative probabilities of states, in the sense of making the following, log-ratio error small:
\begin{equation}\label{eq:err}
\Err (\beta) := \E \left( \frac12 \log \left( \frac{\pi(X)}{\pi(Y)} \middle/ \frac{\nu_{\beta} (X)}{\nu_{\beta} (Y)}\right)^2 \right).
\end{equation}
Note that the exponent in the expectation modifies the logarithm, not its argument, and the expectation is with respect to $X$ and $Y$, which are independent and distributed according to $\P$. In fact, while $\Err(\beta)$ is distinct from the mean squared error that is used to fit the lines in Fig.~\ref{correlation example figure}, its minimizer $\beta^\ast$ coincides with the slopes of these lines (Proposition~\ref{l to r}). Our final result therefore explains that these slopes are close to $1$ when $\wh\rho$ is close to $0$.

\begin{theorem}\label{mse}
The log-ratio error $\Err(\beta)$ is minimized by $\beta^\ast = 1 + \wh \rho r$ and $\Err(\beta^\ast)$ equals
\[
\left(1-\rho^2\right) \Var (\log {\pi} (X)) = \left(1-{\wh \rho\,}^2\right) \Var (\log {\wh \pi} (X)).
\]
\end{theorem}

Theorem~\ref{mse} states that the local distribution $\nu_\beta$ that best approximates $\pi$ in the sense of preserving its relative probabilities is $\nu_{\beta^\ast}$. While the Markov chains with random rates in Fig.~\ref{correlation example figure} have $\beta^\ast \approx 1$, many familiar Markov chains have $\beta^\ast \not\approx 1$ and even $\beta^\ast < 0$. For example, consider the random energy model (REM) on the $N$-dimensional hypercube, where $x_i \in \{0,1\}$ denotes the $i$\textsuperscript{th} coordinate of state $x \in \{0,1\}^N$ \cite{derrida_random-energy_1980,derrida_random-energy_1981}. The REM assigns i.i.d.\ random energies $E(x)$ to the states, which determine a Boltzmann probability distribution~$\pi(x) \propto \exp(-E(x))$. The two Markov chains defined by the rates $Q(x,y) = \exp(E(x))$ and $Q(x,y) = \exp(-E(y))$ are in detailed balance with $\pi$, as are the chains determined by the weighted geometric means of these rates, defined for~$\lambda \in [0,1]$ by $Q (x,y) = \exp (\lambda E(x) -(1-\lambda) E(y))$. (Markov chains of this kind are known as Glauber dynamics \cite{mathieu_convergence_2000}.) As $\lambda$ ranges from~$0$ to $1$, the correlation $\rho$ increases from $0$ to $1$, while the slope~$\beta^\ast$ remains nonnegative (Fig.~\ref{rem example figure}, red points). In contrast, under the Sherrington--Kirkpatrick (SK) model, i.i.d.\ random couplings $g_{ij}$ of the coordinates determine the energies~$E(x)~=~\sum_{i,j} g_{ij} x_i x_j$ of each state \cite{sherrington_solvable_1975}. Due to the dependence of nearby states' energies, the correlations and slopes take negative values when $\lambda$ is less than roughly $0.4$ (Fig.~\ref{rem example figure}, blue points).

\begin{figure}[htbp]
\centering
\includegraphics[width=\textwidth]{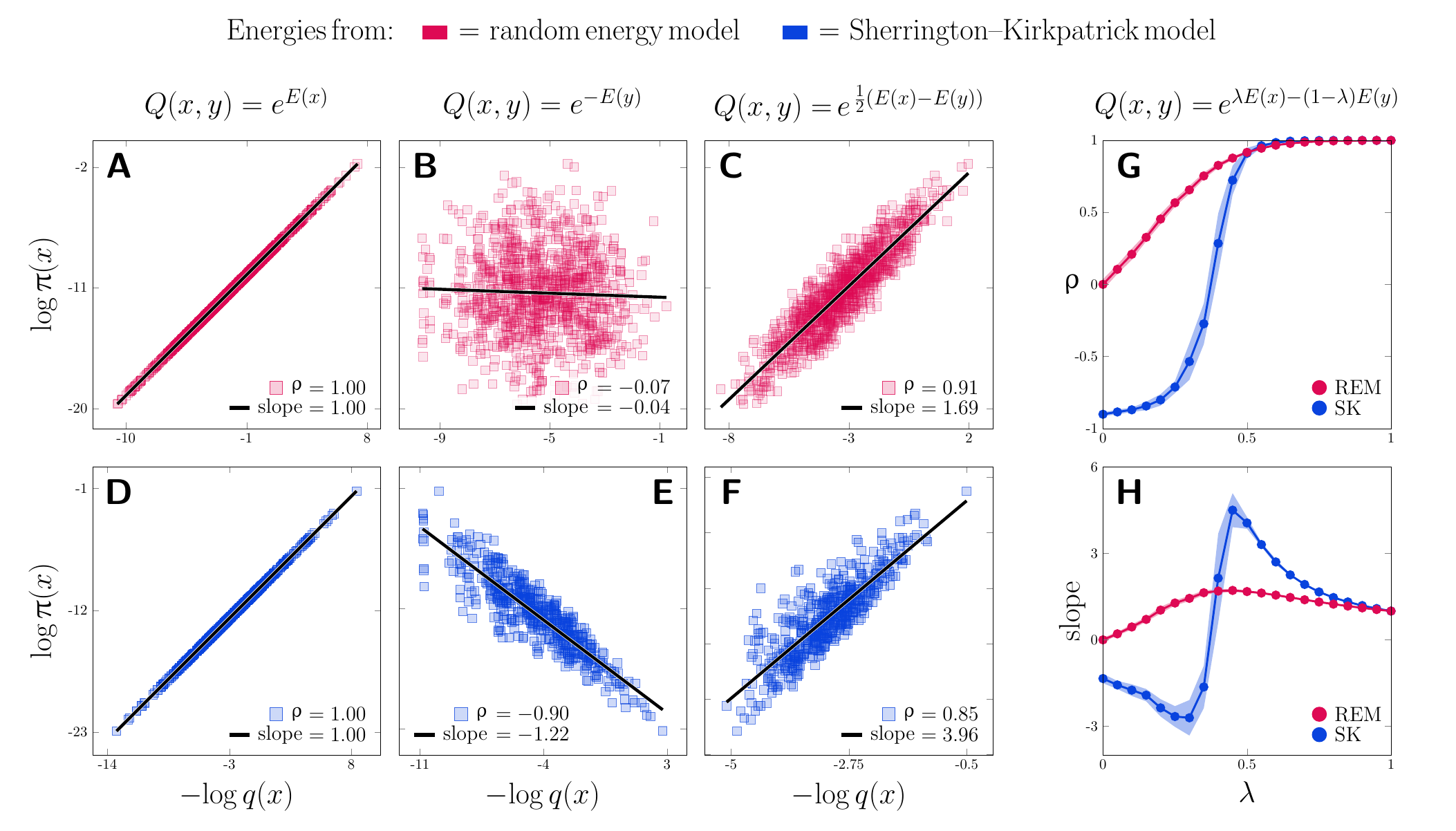}
\caption{The relationships between exit rates and stationary probabilities for spin glasses with reversible dynamics. Each Markov chain has $2^{10}$ states corresponding to the vertices of an $10$-dimensional hypercube, random rates satisfying detailed balance, a Boltzmann stationary distribution given by $\pi(x) \propto \exp (-E(x))$, where the rates and stationary distribution depend on energies $E(x)$ assigned to each state. In (A--C), these energies are given by the random energy model (REM)---namely, i.i.d.\ normal random variables with a mean of $0$ and a variance of $10$. In (D--F), the energies are those of the Sherrington--Kirkpatrick (SK) model, which are correlated through random couplings~$g_{ij}$, according to $E(x) = \sum_{i,j} g_{ij} x_i x_j$, with couplings that are i.i.d.\ normal random variables with a mean of $0$ and a variance of $1/10$. (A) We can see that the first set of rates has exit rates that exactly satisfy a linear relationship between the logarithms of $\pi (x)$ and $1/q(x)$. (B) The other extreme, transition rates that depend on the energy of the destination, produce no detectable linear relationship. In (C), the geometric mean of the first two rates produces high correlation, with a slope greater than $1$. (D) For the SK model, the logarithms of $\pi (x)$ and $1/q(x)$ are exactly collinear, as in (A). (E) However, because adjacent states $x, y$ have correlated energies under the SK model,~$\log \pi(x)$ and~$\shortminus\!\log q(x)$ can be highly negatively correlated, unlike in (B). (F) The geometric mean of the rates again produces high, positive correlation, but with an even greater slope. Finally, (G, H) show the transition from local-global anti-correlation to correlation occurs in the SK model when the weighting $\lambda$ in the geometric mean of the extreme rates is roughly $0.4$.}
\label{rem example figure}
\end{figure}

Theorems~\ref{corr}~and~\ref{mse} characterize the relationship between the exit rates and stationary probabilities of a Markov chain, in terms of $\wh\rho$ and $r$. We can use these results to better understand how the local-global relationship changes as the parameters of the chain vary. Fig.~\ref{contour fig} demonstrates this idea by comparing two models of ant colony behavior that undergo stochastic bifurcations as colony size varies. The F{\"o}llmer--Kirman (FK) model describes the number of ants that choose one of two identical paths to a food source, under the effects of random switching and recruitment \cite{kirman_ants_1993,biancalani_noise-induced_2014,holehouse_non-equilibrium_2022}. It predicts that, as colony size $N$ increases, the colony abruptly transitions from alternately concentrating on one path to splitting evenly between the two. The Beekman--Sumpter--Ratnieks (BRS) model concerns the number of ants that follow and reinforce a pheromone trail, by spontaneously finding it or by being recruited to it \cite{beekman_phase_2001,sumpter_nonlinearity_2003}. It predicts that there is a critical colony size above which the colony sustains a trail, but below which it does not. Despite the apparent similarities of the FK and BRS models, the relationships between their local and global parts change in entirely different ways as colony size increases, as the quantities $\wh\rho$ and $r$ show.

\begin{figure}[htbp]
\centering
\includegraphics[width=\textwidth]{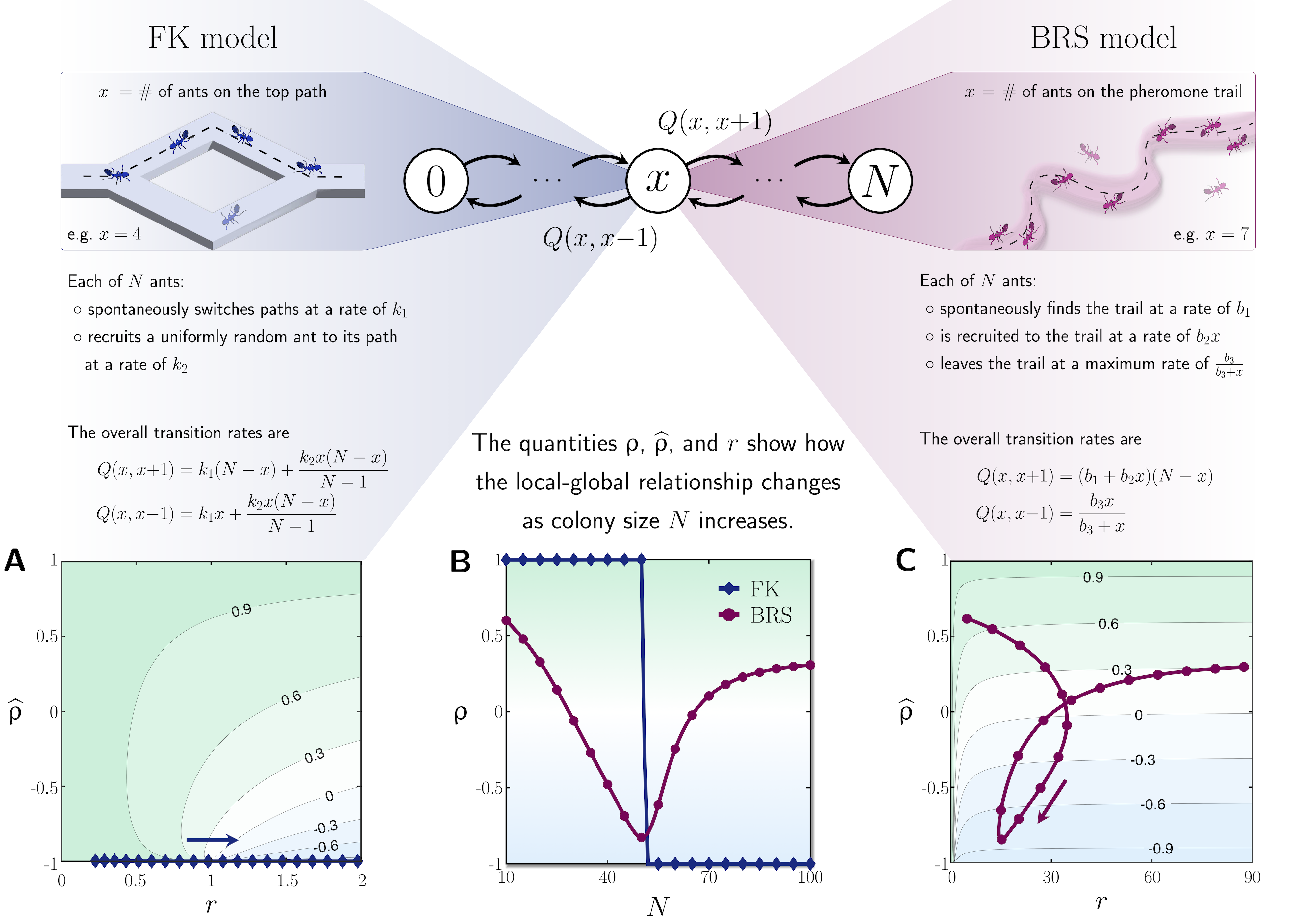}
\caption{Stochastic bifurcations in two models of ant collective behavior, characterized in terms of $\rho$, $\protect\wh\rho$, and $r$. The F{\"o}llmer--Kirman (FK) and Beekman--Sumpter--Ratnieks (BRS) models are Markov chains with nearest-neighbor jumps on the state space $\{0,1,\dots,N\}$. (A) The local and global parts of the FK model's stationary distribution~$\pi$ are essentially perfectly anticorrelated for every $N$. However, as $N$ increases, the global part begins to vary more than the local part, which is reflected by $r$ going from below $1$ to above $1$, and which causes the correlation $\rho$ between $\log \pi (x)$ and $\shortminus\!\log q(x)$ to abruptly change from $1$ to~$-1$, as shown in (B). (C) In contrast, the BRS model generally has $r$ much larger than $1$, so the correlation $\rho$ between $\log \pi(x)$ and $\shortminus\!\log q(x)$ is essentially determined by the correlation $\protect\wh\rho$ between the local and global parts. The correlation $\protect\wh\rho$ decreases as~$N$ increases to the bifurcation point, but recovers to a moderate value for larger $N$. The same is true of $\rho$ (B). (A--C) To make the bifurcation points of both models approximately $N=50$, we used the parameters $(k_1,k_2) = (0.02,1)$ and~$(b_1,b_2,b_3)~=~(0.03,0.002,2)$.}
\label{contour fig}
\end{figure}

Under the FK model, $\rho$ sharply decreases from $1$ to~$-1$, as the colony size $N$ increases through the bifurcation point (Fig.~\ref{contour fig}B). The abrupt transition from correlation to anti-correlation in fact arises from a gradual increase in $r$, the variance of the steady state's global part relative to its local part (Fig.~\ref{contour fig}A). The slope $\beta^\ast$ decreases linearly during this transition, as Theorem~\ref{mse} predicts (Fig.~\ref{fig:b4}A). The exit rates and steady state of the BRS model are also highly anti-correlated in the vicinity of the critical colony size (Fig.~\ref{contour fig}B). However,~$\rho$ changes more gradually and the correlation recovers to a positive value at larger colony sizes. Fig.~\ref{contour fig}C shows that, in contrast to the FK model, the variation in $\wh\rho$ drives the decrease in $\rho$ near the critical colony size. Since $r$ is much larger than $1$ for all but the smallest colonies, its variation matters little to $\rho$, which is reflected by the flatness of the contours in Fig.~\ref{contour fig}C. Instead, the variation of $r$ greatly affects $\beta^\ast$, which takes values as low as roughly $-15$ and as high as roughly $30$, as the colony size increases (Fig.~\ref{fig:b4}B). We note that the results in Fig.~\ref{contour fig} concern a uniformly random state $X$. The curves of the BRS model differ substantially when $X \sim \pi$, but a sharp decrease in $\rho$ near the bifurcation point arises in largely the same way (Figs.~\ref{fig:b5}~and~\ref{fig:b6}).

\section{Discussion}
 
The rattling heuristic, like the Boltzmann distribution, relates a local property of configurations to their weights in the global steady state. Its virtue is that rattling is simple to estimate, only requiring short observations of the dynamics from a configuration \cite{chvykov_low_2021}. However, determining the accuracy of rattling, like detecting the failure of detailed balance, generally requires global information about the dynamics. Practical approaches to the latter are the focus of recent work \cite{battle_broken_2016,rupprecht_fresh_2016,gnesotto_broken_2018,li_quantifying_2019,martinez_inferring_2019,lynn_broken_2021}; we anticipate parallel developments for rattling.

The conceptual heart of rattling theory is the simple observation that many systems tend to spend more time in configurations that they exit more slowly. While systems can exhibit this tendency regardless of whether detailed balance holds, it fails to be universal because the time that a system spends in a configuration depends on both the duration and the frequency of its visits, which can oppose one another. The former is a local property of a configuration, while the latter is generally a global property of the system. In this sense, the failure of the rattling heuristic requires ``adversarial global structure,'' as Chvykov et al.\ claim \cite{chvykov_low_2021}. Theorem~\ref{corr} implies a precise version of this claim: $\rho$ is positive unless the local and global parts of the stationary distribution are sufficiently anticorrelated, that is, unless $\wh\rho$ is at most $-1/r$. In particular, the rattling heuristic always holds when the local part varies more than the global part ($r < 1$).

Together, $r$ and $\wh\rho$ determine the accuracy of the rattling heuristic as well as the form of the weights that it predicts (Theorems~\ref{corr} and \ref{mse}). We view each pair $(r,\wh\rho)$ as defining a class of chains, say $\mathcal{Q}_{(r,\wh\rho)}$, for which a certain local-global principle applies (Fig.~\ref{mctt figure}). The simplest class $\mathcal{Q}_{(0,0)}$ consists of chains with uniform global part, which therefore satisfy $\rho = \beta = 1$ and $\pi(x) \propto 1/q(x)$. Theorem~\ref{thm b} explains that it is simple to make these chains by scaling the entries of doubly-stochastic matrices (Fig.~\ref{ds example figure}). More generally, the class~$\mathcal{Q}_{(r,0)}$ consists of chains with uncorrelated local and global parts, hence they satisfy $\rho = 1/\sqrt{1+r^2}$ and $\beta = 1$, and the rattling heuristic predicts their weights to be $1/q(x)$. As Fig.~\ref{correlation example figure} shows, many chains with random rates essentially belong to this class. The chains in $\mathcal{Q}_{(\cdot,0)}$ are analogous to the systems that Chvykov et al.\ consider to be the domain of rattling theory \cite{chvykov_low_2021}. However, the rattling heuristic also perfectly predicts the weights $1/q(x)^\beta$ of chains in the class $\mathcal{Q}_{(r,\pm 1)}$, with an exponent $\beta = 1\pm r$ that is generally not of order~$1$ and which can even be negative. Figs.~\ref{rem example figure} and \ref{contour fig} include relevant examples.



We define the correlation $\rho$ and other key quantities in terms of a random state $X$, which can have any distribution that makes $\pi (X)$ and $q(X)$ non-constant (see Eq.\ \ref{corr def}). The preceding discussion and the examples in Figs.~\ref{correlation example figure}--\ref{contour fig} emphasize the case when $X$ is uniformly random, primarily because $\rho$ then coincides with the correlation that Chyvkov et al.\ sought to explain \cite{chvykov_low_2021}. However, Theorems~\ref{thm a}--\ref{mse} apply for any permissible distribution on $X$. It may seem natural, for example, to instead choose $X \sim \pi$, so that configurations with greater steady-state weight matter more. This choice leads to different results in some cases, like Fig.~\ref{correlation example figure}D and Fig.~\ref{contour fig} (see Figs.~\ref{fig:b3}D and \ref{fig:b5} for comparison). Nonetheless, the role of $X$ is to model which configurations of a system an observer sees when seeking to predict the steady-state weights using the rattling heuristic (Figure~\ref{experiment figure}). In this scenario, $\pi$ is unknown to the observer, hence $X \sim \pi$ may be unrealistic.



Beyond the application of rattling theory to characterize the local-global relationship in Markov chains (Figure~\ref{contour fig}), our results suggest exploring the relationship between the property of having local weights and central topics of Markov chain theory, like metastability \cite{bovier_metastability_2015,landim_metastable_2019}. Although much of this work concerns discrete-time Markov chains, our results essentially apply to these chains as well, so long as they are irreducible and have positive diagonal entries, because these chains can be ``continuized'' to obtain closely related, continuous-time chains with finite exit rates \cite{aldous_reversible_2002}.

\section{Proofs}

We prove Theorems~\ref{thm a}--\ref{mse} in this section. These results stem from the standard fact that the fraction of time~$\pi (x)$ that a Markov chain $Q$ spends in state $x$ in the long run is proportional to the fraction of visits~$\wh\pi(x)$ that it makes to state $x$, multiplied by the expected duration $1/q(x)$ of each visit. Throughout this section, we continue to assume that $Q$ is an irreducible Markov chain on a finite number $N$ of states.

\begin{proposition}\label{rel-lemma}
The stationary distributions of a Markov chain $Q$ and its jump chain $\wh Q$ are related by  
\begin{equation}\label{pi propto psi}
\pi(x) = \frac{\wh\pi(x)/q(x)}{\sum_y \wh\pi(y)/q(y)}.
\end{equation}
\end{proposition}

\begin{proof}
Since $\wh\pi$ is the stationary distribution of $\wh Q$, by the definition of the jump chain, $\wh\pi$ satisfies
\[
\wh\pi(x) = \sum_{y \neq x} \wh\pi(y) \wh Q(y,x) = \sum_{y \neq x} \frac{\wh\pi(y)}{q(y)} Q(y,x).
\]
In other words, $\wh \pi/q$ is an invariant measure for $Q$:
\[
(\wh\pi/q) Q = \mathbf{0}.
\]
Normalizing $\wh\pi/q$ therefore gives the stationary distribution $\pi$.
\end{proof}

\begin{proof}[Proof of Theorem~\ref{thm a}]
Recall that we are given $k \geq 1$ such that 
\begin{equation}\label{eq:within k of unif}
1/(kN) \leq \wh\pi(x) \leq k/N,
\end{equation}
for every state $x$, and we must show that $\pi(x)$ and $\nu_1 (x)$ are within a factor of $k^2$: 
\[
k^{-2} \leq \pi(x)/\nu_1(x) \leq k^2.
\]
For the upper bound, we use Eq.\ \ref{pi propto psi} to write $\pi(x)$ in terms of $\wh\pi(x)$ and then use Eq.\ \ref{eq:within k of unif} twice to find that
\begin{equation}\label{eq:pi to nu1}
\pi (x) = \frac{\wh\pi(x)/q(x)}{\sum_y \wh\pi(y)/q(y)} \leq k^2 \frac{1/q(x)}{\sum_y 1/q(y)} = k^2 \nu_1 (x).
\end{equation}
The lower bound follows in a similar way.
\end{proof}

The proof of Theorem~\ref{thm b} combines a perturbation bound for Markov chain stationary distributions with Theorem~\ref{thm a}.

\begin{proof}[Proof of Theorem~\ref{thm b}]
We are given a doubly-stochastic matrix $D$ of the same size as $\wh Q$, and a number $k \geq 1$ such that $\wh Q$ and $D$ are within a factor of $k$, and we must show that $\pi$ and $\nu_1$ are within a factor of $k^{2N}$. We use Theorem~1 of \cite{ocinneide_entrywise_1993} to relate the ratio of the entries of $\wh Q$ and $D$ to the ratio of their stationary probabilities. It states that, if irreducible stochastic matrices on $N$ states are within a factor of $k$, then their stationary distributions are within a factor of $k^N$. This result applies to $\wh Q$ and $D$ because we assumed that~$\wh Q$ is irreducible and that $\wh Q$ and $D$ are within a factor of $k$, hence $D$ is also irreducible. Since $D$ is doubly stochastic, it has a uniform stationary distribution. Theorem~1 of \cite{ocinneide_entrywise_1993} and the fact that $\wh Q$ and $D$ are within a factor of $k$ therefore imply that $\wh\pi$ and the uniform distribution are within a factor of $k^N$. Using Theorem~\ref{thm a}, we conclude that $\pi$ and $\nu_1$ are within a factor of $k^{2N}$.
\end{proof}

The next proof uses the basic fact that, for any two random variables $U$ and $V$ with positive, finite variances, the correlation of $U$ and $U+V$ equals
\[
\Corr (U,U+V) = \frac{\Var (U) + \Cov (U,V)}{\sqrt{\Var (U) \left( \Var (U) + 2 \Cov (U,V) + \Var (V) \right)}}.
\]
With some algebra, we can rewrite this formula in terms of~$s~=~\Corr (U,V)$ and $t^2 = \Var (V) / \Var (U)$ as
\begin{equation}\label{eq:stcorr}
\Corr (U,U+V) = \frac{1+st}{\sqrt{1+2st + t^2}}.
\end{equation}

Recall that $\rho$ denotes the correlation between the logarithms of $\pi (X)$ and $1/q(X)$ for a random state $X \sim \P$. The proof of Theorem~\ref{corr} is a direct calculation of $\rho$, using Eq.\ \ref{pi propto psi}.

\begin{proof}[Proof of Theorem~\ref{corr}] 
By Eq.\ \ref{pi propto psi}, the probability $\pi (X)$ satisfies
\[
\log \pi(X) = - \log q(X) + \log \wh\pi(X) - \log \left( \sum_y \wh\pi(y)/q(y) \right).
\]
Since the last term is constant, it does not affect the correlation $\rho$, which  therefore equals 
\begin{equation*}
\Corr (-\!\log q (X), -\!\log q (X) + \log \wh\pi (X)) = \Corr (U, U+V)
\end{equation*}
in terms of $U = -\!\log q (X)$ and $V = \log \wh\pi (X)$. We then obtain the formula for $\rho$ from Eq.\ \ref{eq:stcorr}, where $s = \wh\rho$ and $t = r$.
\end{proof}

Recall the definition of the log-ratio error $\Err(\beta)$ in Eq.\ \ref{eq:err}. The proof of Theorem~\ref{mse} entails a calculation of~$\Err(\beta)$, using Eq.\ \ref{pi propto psi}.

\begin{proof}[Proof of Theorem~\ref{mse}]
    The log-ratio error $\Err(\beta)$ equals
    \[
    \frac12 \E \left( (A - B)^2 \right) = \frac12 \left( \E \left(A^2\right) - 2 \,\E(AB) + \E \left(B^2\right) \right) = \Var (A),
    \]
    in terms of the auxiliary variables
    \[
    A = \log \left(\pi(X) \middle/ \nu_\beta (X)\right) \quad \text{and} \quad B=\log \left(\pi(Y) \middle/ \nu_\beta (Y)\right), 
    \]
    due to the linearity of expectation and because $A$ and $B$ are i.i.d. 
    By Eq.\ \ref{pi propto psi}, the variance of $A$ equals
    \[
    \Var \left( \log \wh\pi (X) + (\beta-1) \log q(X) \right) = \Var(V-(\beta-1)U),
    \]
    where $U = -\!\log q(X)$ and $V = \log\wh\pi (X)$. We write the variance as
    \[
    \Var (V) - 2(\beta-1) \Cov(U,V) + (\beta-1)^2 \Var(U),
    \]
    and then identify factors of $\wh\rho$ and $r$, to find that
    \[
    \Err(\beta) = \left(r^2-2\wh\rho r (\beta-1) + (\beta-1)^2\right) \Var (U).
    \]
    The error is minimized by $\beta^\ast=1+\wh\rho r$, which satisfies
    \[
    \Err(\beta^\ast) = r^2 \left(1-{\wh\rho}^{\,2}\right) \Var (\log q (X)) = \left(1-{\wh\rho}^{\,2}\right) \Var (\log \wh\pi (X)).
    \]
    A simple but tedious manipulation of this equation using Eq.\ \ref{main est} and Eq.\ \ref{pi propto psi} shows that $\Err(\beta^\ast)$ further equals $(1-\rho^2) \Var(\log\pi(X))$. See Appendix~\ref{appendix a} for details.
\end{proof}






\newpage

\appendix
\renewcommand\thefigure{\thesection.\arabic{figure}}
\setcounter{figure}{0}

\section{Supplementary text}\label{appendix a}

\subsection{The Markov chain analogue of rattling}

Chvykov et al.\ \cite{chvykov_low_2021} use the phenomenon of thermophoresis to motivate their definition of rattling. Thermophoresis entails a Brownian particle diffusing through a temperature landscape, with a diffusivity that varies with the temperature. In this case, Chvykov et al.\ define the rattling $R(x)$ at a position $x$ in continuous space as the logarithm of the diffusivity $D(x)$. In one spatial dimension, the diffusivity at $x$ is the rate at which the mean squared displacement from $x$ grows over a short period of time $t$. In terms of the position $X(t)$ of the particle at time $t$, and the average $\langle \,\cdot\, \rangle_{x}$ over many trajectories starting at $X(0)=x$, the diffusivity equals
\[
D(x) = \lim_{t \to 0} \frac{1}{t} \left\langle (X(t) - X(0))^2 \right\rangle_{x}.
\]

To identify an analogous quantity when $x$ is a discrete state of a Markov chain $X(t)$, we instead define the displacement of the chain at time $t$ to be the indicator $\1(A_{t})$ of the event $A_{t} = \{\text{{$X(s) \neq X(0)$ for some $s \leq t$}}\}$. In other words, the displacement equals $1$ if the chain exits the initial state before time $t$, and $0$ otherwise. The corresponding diffusivity equals
\[
D(x) = \lim_{t\to 0} \frac{1}{t} \E_x \left( \1(A_{t})^2 \right) = \lim_{t\to 0} \frac{1}{t} \E_x \left( \1(A_{t}) \right) = \lim_{t\to 0} \frac{1}{t} \P_x(A_{t}) = \lim_{t\to 0} \frac{1}{t} \left( 1 - e^{-tq(x)} \right) = q(x),
\]
in terms of the expectation $\E_x$ and probability $\P_x$ with respect to the distribution of the Markov chain, starting from $X(0) = x$. The second and third equalities hold by properties of indicator random variables; the third holds because the time at which the chain first exits $x$ is an exponential random variable with a mean of $1/q(x)$; the fourth follows from the Taylor series $e^u = 1 + u + O(u^2)$. Since Chvykov et al.\ define rattling by $R(x) = \log D(x)$ in the case of thermophoresis, a natural Markov chain analogue of rattling is $\log q(x)$.

\subsection{A lower bound of \texorpdfstring{$\rho$}{p} that makes no reference to \texorpdfstring{$\protect\wh\rho$}{phat}}

The following bound is a simple consequence of Theorem~\ref{corr}, which shows that, for the correlation $\rho$ to be positive, it suffices for $r$ to be less than $1$.

\begin{corollary}\label{cor:a1}
    If $r<1$, then 
    \[
    \rho \geq \sqrt{\frac{1-r}{1+r}}.
    \]
\end{corollary}

\begin{proof}
    In terms of the quantity $\eps = (1-\wh\rho) r / (1+r)$, the correlation satisfies
    \[
    \rho = \frac{1+\wh\rho r}{\sqrt{1+2\wh\rho r + r^2}} = \frac{1-\eps}{\sqrt{1 - 2\eps/(1+r)}} \geq \sqrt{1-\eps} = \sqrt{\frac{1+\wh\rho r}{1+r}} \geq \sqrt{\frac{1-r}{1+r}}.
    \]
    The equalities and inequalities are respectively due to Theorem~\ref{corr}; the choice of $\eps$; the assumption that $r < 1$, which implies that $2\eps/(1+r)$ is less than $\eps$, and that replacing the former by the latter gives a lower bound because the numerator and denominator are both positive; the choice of $\eps$, again; and the fact that $\wh\rho \geq -1$. 
\end{proof}

\subsection{A bound on the relative error of approximating \texorpdfstring{$\pi$}{pi} by \texorpdfstring{$\nu_{\beta^\ast}$}{nu beta ast} }

Recall that we define the log-ratio error $\Err(\beta)$ for $\beta \in \R$ to be
\[
\Err (\beta) := \E \left( \frac12 \log \left( \frac{\pi(X)}{\pi(Y)} \middle/ \frac{\nu_{\beta} (X)}{\nu_{\beta} (Y)}\right)^2 \right),
\]
where expectation is taken with respect to the joint distribution of independent random variables $X \sim \P$ and $Y \sim \P$. 
In this section, for the sake of simplicity, we assume that $\P$ is the uniform distribution on $N$ states.

Unlike Theorems~\ref{thm a}~and~\ref{thm b}, which bound above the ratios $\pi/\nu_1$ and $\nu_1/\pi$, Theorem~\ref{mse} compares $\pi$ and $\nu_{\beta^\ast}$ through $\Err(\beta^\ast)$. The ratios $\pi/\nu_{\beta^\ast}$ and $\nu_{\beta^\ast}/\pi$ can be large for some states even when $\Err(\beta^\ast)$ is small because the latter error concerns the average squared differences of the logarithms of $\pi$ and $\nu_{\beta^\ast}$. Nevertheless, it is possible to bound above the ratios $\pi/\nu_{\beta^\ast}$ and $\nu_{\beta^\ast}/\pi$ in terms of $\Err(\beta^\ast)$.

\begin{proposition}\label{prop:l1 l2 bd}
If $X$ is a uniformly random state, then $\pi$ and $\nu_{\beta^\ast}$ are within a factor of 
\[
\exp \left( \sqrt{2N\Err(\beta^\ast)} \right).
\]
\end{proposition} 



Note that the following, simple bound on the log-ratio error
\[
\Err (\beta) = \E \left( \frac12 \log \left( \frac{\pi(X)}{\pi(Y)} \middle/ \frac{\nu_{\beta} (X)}{\nu_{\beta} (Y)}\right)^2 \right) \geq \frac{1}{2N^2} \max_{x,y} \log \left( \frac{\pi (x)}{\pi (y)} \middle/ \frac{\nu_\beta (x)}{\nu_\beta (y)} \right)^2
\]
only implies that $\pi$ and $\nu_{\beta^\ast}$ are within a factor of $\exp \left( \sqrt{2N^2 \Err(\beta^\ast)} \right)$.

\begin{proof}[Proof of Proposition~\ref{prop:l1 l2 bd}]
    Denote $\nu = \nu_{\beta^\ast}$. The ratio $\pi (x) /\nu (x)$ equals
    \[
    Z_\nu q(x)^{\beta^\ast} \pi(x) = Z_\nu \exp \left(\log \pi(x) + \beta^\ast \log q(x) \right),
    \]
    where $Z_\nu = \sum_y q(y)^{-\beta^\ast}$.  Using this fact twice gives 
    \[
    \frac{\pi (x)}{\pi (y)} \cdot \frac{\nu (y)}{\nu (x)} = \exp \left(\log \pi (x) - a + \beta^\ast \log q(x) - (\log \pi (y) - a + \beta^\ast \log q(y)) \right)
    \]
    for any $a \in \R$. In terms of  $d_a (x) = \log \pi (x) - a + \beta^\ast \log q(x)$, the ratios satisfy
    \[
        \max\left\{ \frac{\pi (x)}{\pi (y)} \cdot \frac{\nu (y)}{\nu (x)}, \frac{\pi (y)}{\pi (x)} \cdot \frac{\nu (x)}{\nu (y)} \right\} = \exp \left( |d_a (x) - d_a (y)| \right).
    \]
    We bound above $|d_a (x) - d_a (y)|$ using the triangle inequality and the fact that $\max_i v(i) \leq \|v\|_2$ for every $v \in \R^N$: 
    \[
    |d_a (x) - d_a (y)| \leq 2 \max_x d_a (x) \leq 2 \| d_a  \|_2.
    \]
    Because $X$ is uniformly random, $\|d_a \|_2/N$ equals the mean squared error $\E(d_a (X)^2)$ of approximating $\log\pi (X)$ by $a - \beta^\ast \log q(X)$. In the next section, Proposition~\ref{l to r} independently establishes that there is a value of $a$, denoted $a^\ast$, for which $\E(d_{a^\ast} (X)^2) = \Err(\beta^\ast)$. Since the preceding bound holds for any $a$, Proposition~\ref{l to r} implies that
    \begin{equation}\label{eq: pi nu mse bd}
        \max\left\{ \frac{\pi (x)}{\pi (y)} \cdot \frac{\nu (y)}{\nu (x)}, \frac{\pi (y)}{\pi (x)} \cdot \frac{\nu (x)}{\nu (y)} \right\} \leq \exp \left(\sqrt{2N \Err(\beta^\ast)} \right).
    \end{equation}
    Call the quantity in the bound $k = \exp (\sqrt{2N \Err(\beta^\ast)})$. Because $\pi$ and $\nu$ are probability distributions, Eq.\ \ref{eq: pi nu mse bd} implies that
    \[
        \frac{\pi (x)}{\nu (x)} = \sum_y \left( \frac{\nu (y)}{\nu (x)} \cdot \frac{\pi (x)}{\pi (y)} \right) \pi (y) \leq k,
    \]
    as well as the bound $\nu(x)/\pi(x) \leq k$. Hence, $\pi$ and $\nu$ are within a factor of $k$.
\end{proof}

\subsection{The minimizer of \texorpdfstring{$\Err(\beta)$}{L(beta)} is the slope of the line of best fit}

The log-ratio error $\Err(\beta)$ differs from the mean squared error of predicting $\log \pi(X)$ by a linear function of $-\!\log q(X)$, which is the error that governs the linear fits in Figs.~\ref{correlation example figure}~and~\ref{rem example figure}. For an intercept-slope pair $(a,b) \in \R^2$, the mean squared error is
\[
\MSE(a,b) := \E \left( (\log \pi(X) - a + b \log q(X))^2 \right).
\]
Although these errors are different, their minimizers coincide, in the sense of the following result.

\begin{proposition}\label{l to r}
    The minimizers $\beta^\ast$ of $\Err(\beta)$ and $(a^\ast,b^\ast)$ of $\MSE(a,b)$ satisfy $\beta^\ast=b^\ast$. Moreover, $\Err(\beta^\ast) = \MSE(a^\ast,b^\ast)$. 
\end{proposition}

The proof uses the basic fact that, if $U$ and $V$ are random variables with positive, finite variances, then the approximation of $U+V$ by $a + bU$ for $(a,b) \in \R^2$ has a mean squared error of
\[
\E \big( (U+V - a - b U)^2 \big),
\]
which is minimized by $(a^\ast,b^\ast)$ with 
\begin{equation}\label{eq:best slope}
    b^\ast = 1 + \frac{\Cov (U,V)}{\Var (U)}.
\end{equation}
Moreover, the associated error satisfies
\begin{equation}\label{eq:mse mb}
    \MSE (a^\ast,b^\ast) = \Var (V) - (1-b^\ast)^2 \Var (U).
\end{equation}

\begin{proof}[Proof of Proposition~\ref{l to r}]
    By Theorem~\ref{mse}, it suffices to show that $b^\ast = 1 + \wh\rho r$ and $\MSE(a^\ast,b^\ast) = \left(1-\wh\rho^{\,2}\right) \Var (\log \wh\pi(X))$. According to Eq.~\ref{pi propto psi}, the mean squared error satisfies
    \[
    \MSE(a,b) = \E \left( \left(\log \wh \pi (X) - \log q(X) - a + b\log q(X) \right)^2 \right) = \E \big( (U+V - a - bU)^2 \big),
    \]
    in terms of $U = -\!\log q (X)$ and $V = \log \wh\pi (X)$. Hence, by Eq.\ \ref{eq:best slope}, its minimizer $(a^\ast,b^\ast)$ has
    \begin{equation}\label{eq:beta}
    b^\ast = 1 + \frac{\Cov (U,V)}{\Var (U)} = 1 + \wh\rho r.
    \end{equation}
    Concerning the error $\MSE(a^\ast,b^\ast)$, by substituting Eq.\ \ref{eq:beta} into Eq.\ \ref{eq:mse mb}, we find that
    \[
    \left(\frac{\Var(V)}{\Var(U)} - (1-b^\ast)^2\right) \Var (U) = r^2 \left(1-\wh\rho^{\,2}\right) \Var (\log q(X)) = \left(1-\wh\rho^{\,2}\right) \Var (\log \wh\pi(X)).
    \]
\end{proof}

\subsection{The equality of two expressions for the error \texorpdfstring{$\Err(\beta^\ast)$}{L(beta ast)} }

At the end of the proof of Theorem~4, we state that simple but tedious algebra, along with Eq.~\ref{main est} and Eq.~\ref{pi propto psi}, establishes the equality
\[
\left(1-{\wh\rho}^{\,2}\right) \Var (\log \wh\pi (X)) = \left(1-\rho^2\right) \Var (\log \pi (X)).
\]
Indeed, the formula for $\rho$ (Eq.~\ref{main est}) implies that
\[
1 - \rho^2 = 1 - \left(\frac{1+\wh\rho r}{\sqrt{1 + 2 \wh\rho r + r^2}}\right)^2 = \frac{r^2 \left(1 - {\wh\rho}^{\,2}\right)}{1+2\wh\rho r + r^2},
\]
while the expression for $\pi$ (Eq.~\ref{pi propto psi}) gives
\begin{align*}
\Var (\log \pi(X)) &= \Var (\log \wh\pi(X)) + 2 \Cov (\log\wh\pi (X), -\log q(X)) + \Var(\log q(X))\\ 
&= \Var(\log q(X)) \left(\frac{\Var(\log\wh\pi(X))}{\Var(\log q(X))} + \frac{2\Cov(\log\wh\pi(X),-\log q(X))}{\Var(\log q(X))} + 1 \right)\\ 
&= \Var(\log q(X)) \left(r^2 + 2 \wh\rho r + 1\right).
\end{align*}
Together, these two expressions imply that
\[
\left(1-\rho^2\right) \Var (\log \pi (X)) = \frac{r^2 \left(1 - {\wh\rho}^{\,2}\right)}{1+2\wh\rho r + r^2} \cdot \Var(\log q(X)) \left(r^2 + 2 \wh\rho r + 1\right) = \left(1-{\wh\rho}^{\,2}\right) \Var (\log \wh\pi (X)).
\]

\newpage

\section{Supplementary figures}

\begin{figure}[!htbp]
\centering
\includegraphics[width=\textwidth]{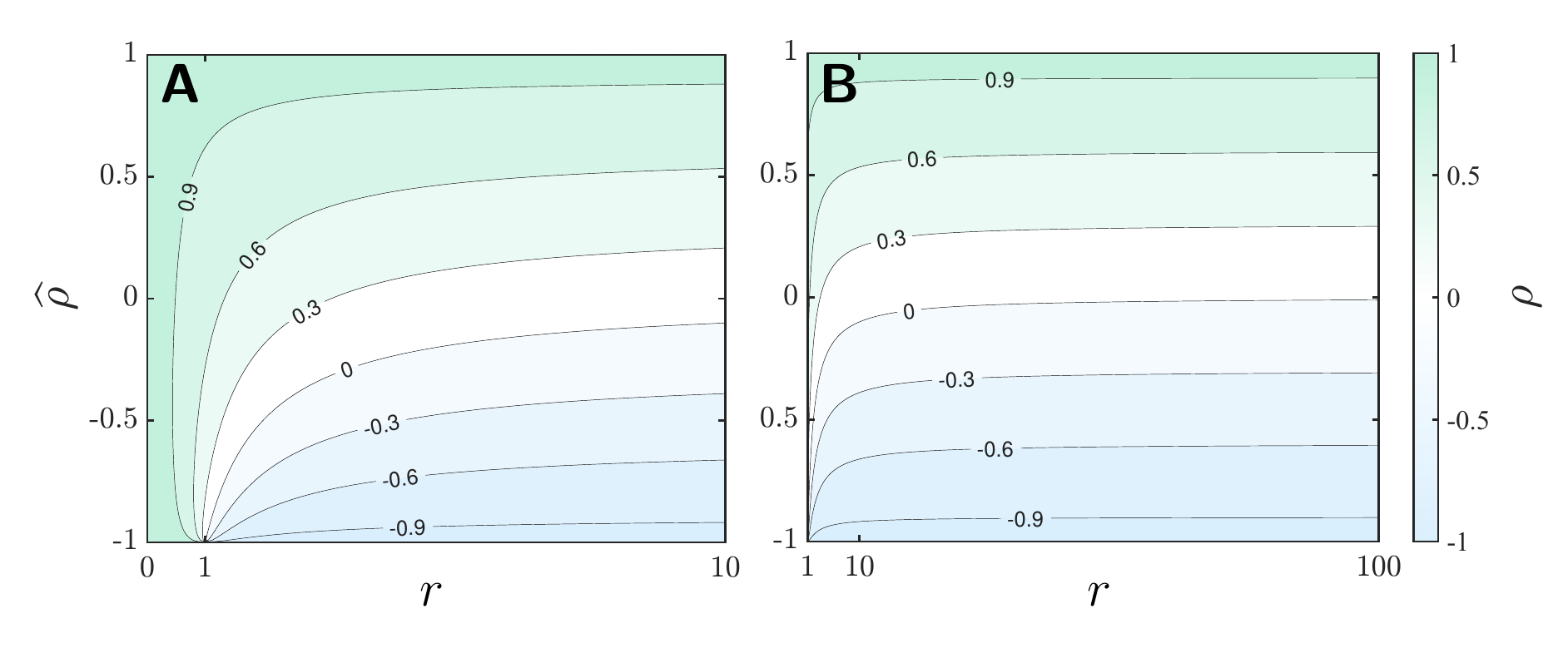}
\caption{Contour plots of $\rho$ in terms of $\protect\wh\rho$ and $r$ for $0 \leq r \leq 10$ (A) and $1 \leq r \leq 100$ (B). (A) When $r \ll 1$, the contours of $\rho$ are vertical lines, which reflects the fact that $r$ primarily determines $\rho$ in this region. (B) When $r \gg 1$, the contours of $\rho$ are horizontal lines, which reflects the fact that $\rho \approx \protect\wh\rho$ in this region.}
\label{fig:b1}
\end{figure}

\begin{figure}[!htbp]
\centering
\includegraphics[width=0.85\textwidth]{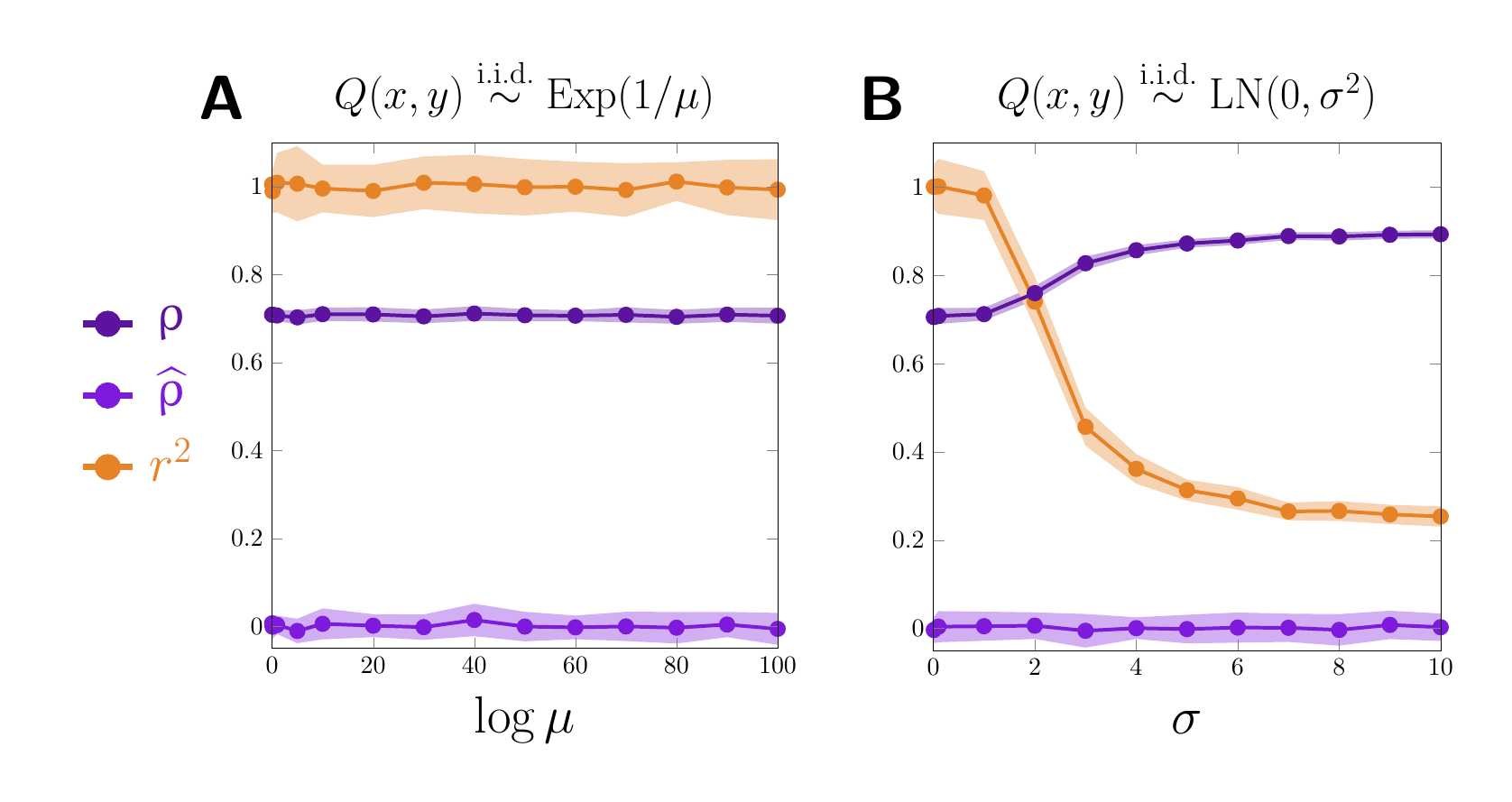}
\caption{Markov chains with random, directed rates. (A, B) The states of the chain correspond to the vertices of the complete graph with $2^{10}$ vertices, with a pair of directed rates connecting every pair of vertices. (A) Independent and identically distributed (i.i.d.) exponential rates produce a chain with values of $\rho$, $\protect\wh\rho$, and $r^2$ that essentially remain constant as the mean (and variance) $\mu$ of the rates varies over many orders of magnitude. (B) In contrast, when the rates are i.i.d.\ log-normal random variables, increasing the standard deviation parameter $\sigma$ results in a decrease of $r^2$ and concomitant increase in $\rho$, while $\protect\wh\rho$ remains relatively constant. Note that the variance of the rates increases exponentially with $\sigma$. (A, B) The plotted lines reflect the mean values of the quantities over $25$ trials; the shaded regions above and below the plotted lines represent $\pm 1$ standard deviation.}
\label{fig:b2}
\end{figure}

\begin{figure}[!htbp]
\centering
\includegraphics[width=\textwidth]{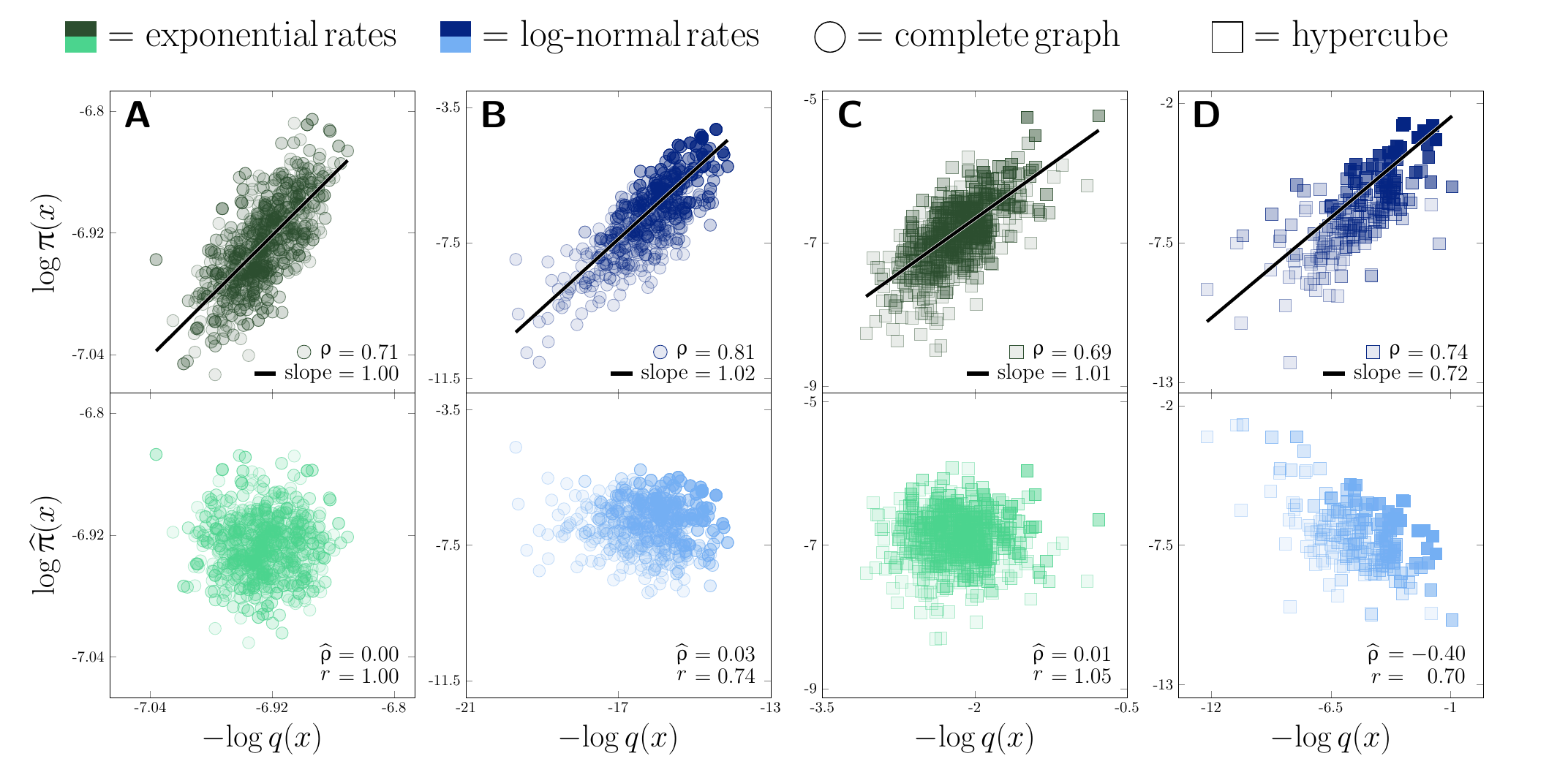}
\caption{A variant of Fig.~\ref{correlation example figure} with the random state $X$ distributed according to $\pi$ instead of uniformly at random.}
\label{fig:b3}
\end{figure}

\begin{figure}[!htbp]
\centering
\includegraphics[width=11.4cm]{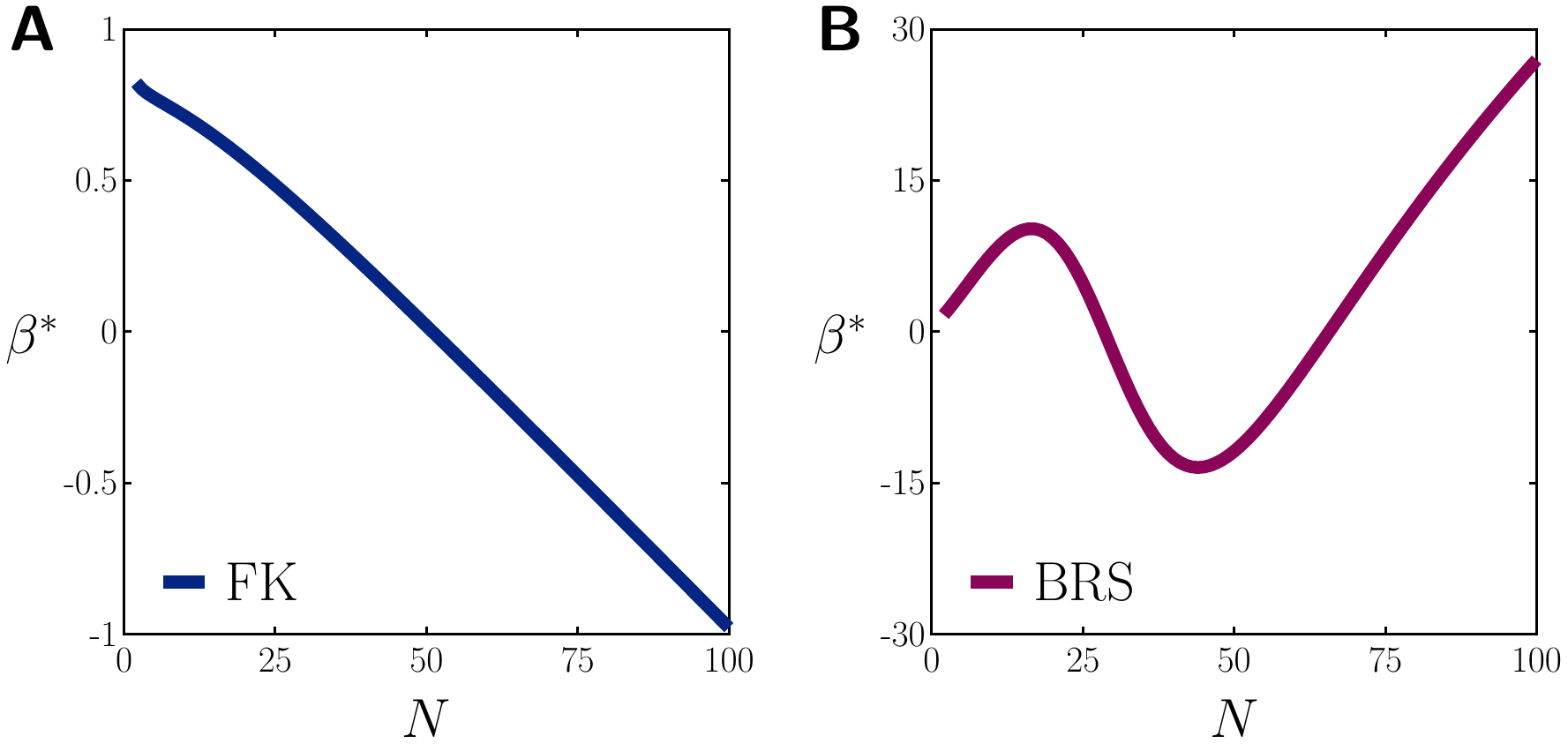}
\caption{Slopes $\beta^{\!\ast}$ versus colony size $N$ for the FK and BRS models.}
\label{fig:b4}
\end{figure}

\begin{figure}[!htbp]
\centering
\vspace{1em}
\includegraphics[width=\textwidth]{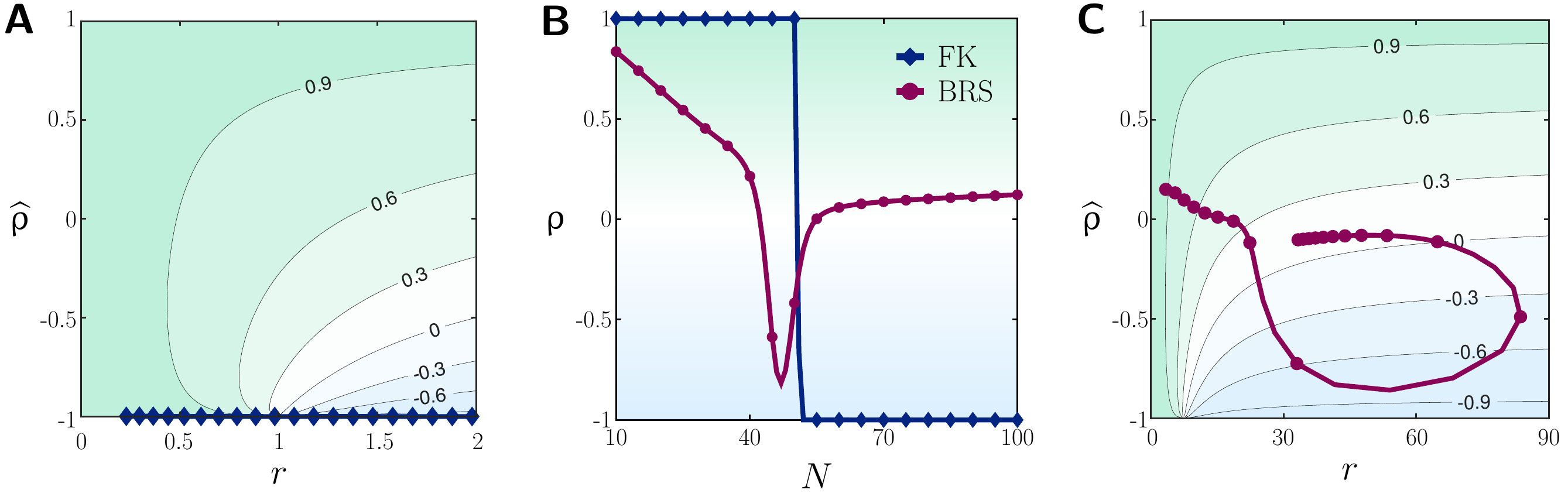}
\caption{Variants of the contour diagrams from Fig.~\ref{contour fig} in which the random state $X$ is distributed according to $\pi$ instead of the uniform distribution on states. (A, B) The plots for the FK model are indistinguishable from those in Fig.~\ref{contour fig}, while (B, C) those of the BRS model are markedly different.}
\label{fig:b5}
\end{figure}

\begin{figure}[!htbp]
\centering
\includegraphics[width=11.4cm]{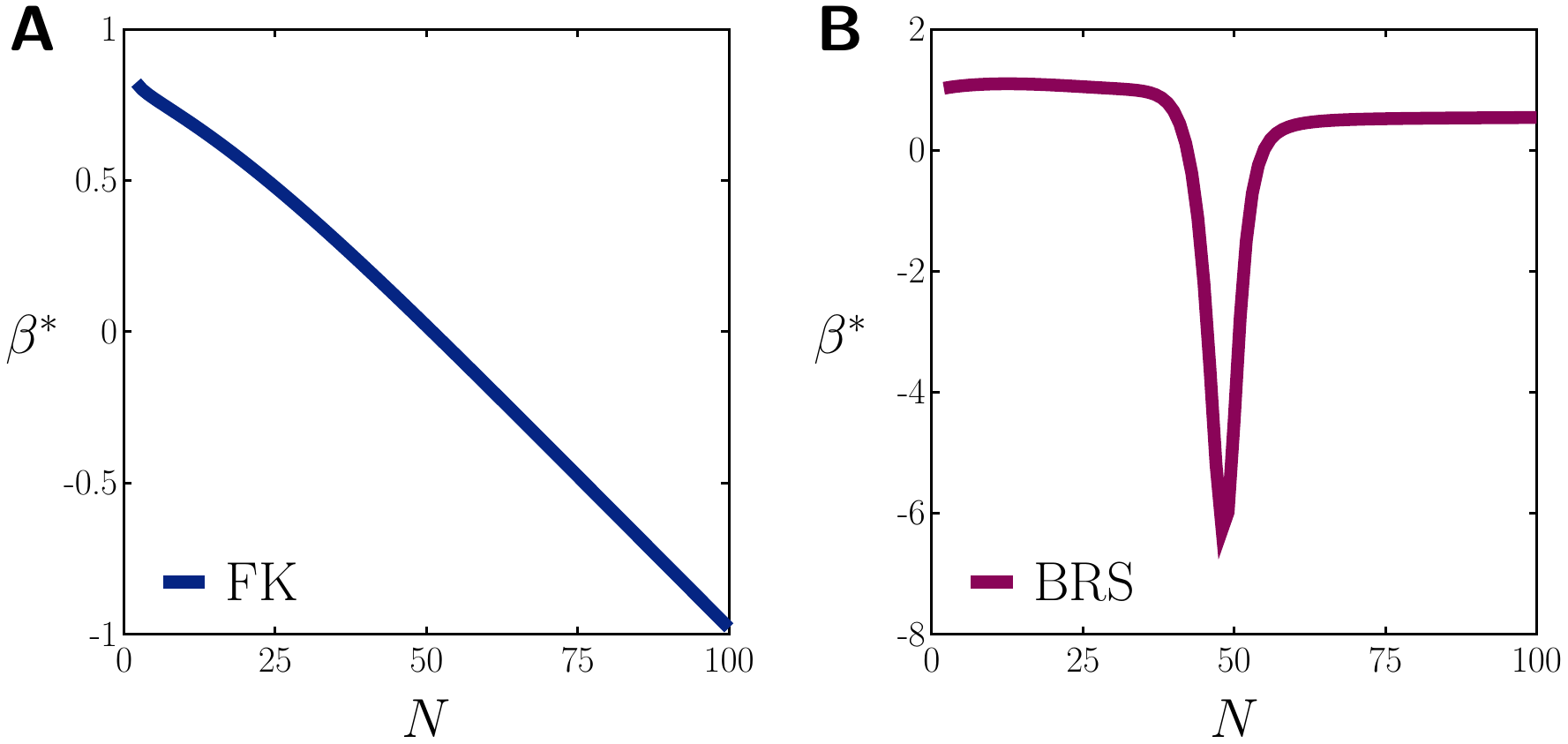}
\caption{Slopes $\beta^{\!\ast}$ versus colony size $N$ for the FK and BRS models, where the random state $X$ is distributed according to $\pi$ instead of the uniform distribution.}
\label{fig:b6}
\end{figure}

\end{document}